\documentclass[11pt]{article}

\usepackage{mathptmx,eucal,amsmath,amscd,amssymb,amsthm,xspace}
\usepackage[all,tips]{xy}
\usepackage[dvips]{graphicx}
\usepackage{graphics,epsfig,wrapfig,verbatim,syntonly}
\usepackage{hyperref,amssymb,color, url,fancyhdr}
\usepackage{lineno}

\theoremstyle{plain}

\newtheorem{thm}{Theorem}[section]

\newtheorem{cor}[thm]{Corollary}

\newtheorem{lemma}[thm]{Lemma}
\newtheorem{sch}[thm]{Scholium}

\theoremstyle{definition}

\newtheorem{ques}[thm]{Question}

\DeclareMathOperator{\GL}{GL}

\DeclareMathOperator{\D}{D}

\DeclareMathOperator{\word}{w}
\DeclareMathOperator{\sub}{s}\DeclareMathOperator{\Farb}{F}

\newcommand{\vp}{\varphi}

\newcommand{\nid}{\noindent}

\newcommand{\iny}{\infty}

\newcommand{\abs}[1]{\left\vert#1\right\vert}
\newcommand{\set}[1]{\left\{#1\right\}}
\newcommand{\norm}[1]{\left\vert \left\vert #1\right\vert\right\vert}

\newcommand{\pr}[1]{\left( #1 \right) }

\newcommand{\lra}{\longrightarrow}

\newcommand{\B}[1]{\ensuremath{\mathbf{#1}}}

\newcommand{\Fr}[1]{\ensuremath{\mathfrak{#1}}}

\newcommand{\N}{\mathbf{N}}
\newcommand{\Q}{\mathbf{Q}}

\newcommand{\Z}{\mathbf{Z}}
\newcommand{\C}{\mathbf{C}}

\fancypagestyle{plain}

\begin{document}
\bibliographystyle{plain}

%-----------------------------------------------------------
%-----------------------------------------------------------

\title{\textbf{Extremal behavior of divisibility functions}}
\author{Khalid Bou-Rabee\thanks{University of Michigan, Ann Arbor, MI 48109. E-mail: \tt{khalidb@umich.edu}} \, and D. B. McReynolds\thanks{Purdue University, West Lafayette, IN 47907. E-mail: \tt{dmcreyno@math.purdue.edu}}}
\maketitle

%-----------------------------------------------------------
%-----------------------------------------------------------
\begin{abstract}
\nid In this short article, we study the extremal behavior $\Farb_\Gamma(n)$ of divisibility functions $\D_\Gamma$ introduced by the first author for finitely generated groups $\Gamma$. These functions aim at quantifying residual finiteness and bounds give a measurement of the complexity in verifying a word is non-trivial. We show that finitely generated subgroups of $\GL(m,K)$ for an infinite field $K$ have at most polynomial growth for the function $\Farb_\Gamma(n)$. Consequently, we obtain a dichotomy for the growth rate of $\log \Farb_\Gamma(n)$ for finitely generated subgroups of $\GL(n,\C)$. We also show that if $\Farb_\Gamma(n) \preceq \log \log n$, then $\Gamma$ is finite. In contrast, when $\Gamma$ contains an element of infinite order, $\log n \preceq \Farb_\Gamma(n)$. We end with a brief discussion of some geometric motivation for this work.
\end{abstract}
%-----------------------------------------------------------
%-----------------------------------------------------------
\section{Introduction}

A group is \emph{residually finite} if the intersection of all the finite index subgroups is trivial. We continue the study of quantifying residual finiteness, started in \cite{B10} and furthered in \cite{B11}, \cite{BM10,BM11}, and \cite{KM}. This venue is concerned with the asymptotic growth of variants of the \emph{normal divisibility function} $\D_\Gamma\colon \Gamma \to \N$ defined by
\[ \D_\Gamma(g) = \min \{ [\Gamma:\Delta] : \Delta \lhd \Gamma, g \notin \Delta \}. \]
The asymptotic or $L^\iny$--behavior of this function is measured by
\[ \Farb_{\Gamma,X}(n) = \max \{ \D_\Gamma(g) : g \in B^\bullet_{\Gamma,X}(n) \}, \]
where $B^\bullet_{\Gamma,X}(n)$ is the ball of radius $n$ minus the identity for $\Gamma$ with respect to some fixed finite generating set $X$. The function $\Farb_{\Gamma,X}(n)$ is related to both the word growth $\word_{\Gamma,X}(n)$ and normal subgroup growth function $\sub_\Gamma(n)$ via a basic inequality established in \cite{BM10} (see (\ref{8}) below).

It is a classical theorem of Mal'cev \cite{M40} that any finitely generated linear group is residually finite \cite{M40}. In \cite[Theorem 1.1]{BM11}, we proved that for finitely generated linear groups, $\Farb_{\Gamma,X}(n) \preceq (\log(n))^r$ for some $r>0$ if and only if $\Gamma$ is virtually nilpotent; a similar result with restrictions on finite quotients was established in \cite[Theorem 2]{B11} without the linearity assumption. In \cite[Theorem 0.1]{B10} and the substantial generalization \cite[Theorem 1.3]{BK12}, the growth rate of the function $\Farb_{\Gamma,X}(n)$ was established for a large class of arithmetic lattices. Our first main result completes our goal of determining the growth of $\Farb_\Gamma(n)$ for finitely generated linear groups $\Gamma$. Specifically, we prove the following:

\begin{thm}\label{PolyBound}
Let $\Gamma$ be a finitely generated subgroup of $\GL(m,K)$, where $K$ is an infinite field. Then $\Farb_{\Gamma,X}(n) \preceq n^d$ for some $d$ depending only on $m$ and $K$.
\end{thm}

The chief difficulty in proving Theorem \ref{PolyBound} versus what was done in \cite[Theorem 0.1]{B10} and \cite[Theorem 1.3]{BK12} (also the general methods used in \cite{BM10}) is the possibility that the field of coefficients for the group is transcendental over $\Q$ or $\mathbf{F}_p$. Geometrically, this issue is dealt with via a deformation of the representation in the variety of representations to a representation with coefficients in $\overline{\Q}$ or $\overline{\mathbf{F}_p}$ since such a point cannot be locally rigid by work of Weil; the resulting representation need not be faithful but a fixed non-trivial word will have non-trivial image generically. Algebraically, this deformation equates to employing evaluation maps on function fields to the field of coefficients. We will take the geometrically less intuitive algebraic approach here as it is better suited for quantitative analysis. Combining Theorem \ref{PolyBound} with \cite[Theorem 0.2]{B10} and \cite[Theorem 1.1]{BM11}, we have the following dichotomy which was a main goal of the study of the function $\Farb_{\Gamma,X}(n)$.

\begin{cor}\label{CorPolyBound}
Let $\Gamma$ be a finitely generated subgroup of $\GL(m,\C)$. Then there exists a positive integer $b$ such that 
\begin{itemize}
\item[(i)]
$\Farb_{\Gamma,X}(n) \preceq (\log n)^b$, or
\item[(ii)]
$\Farb_{\Gamma,X}(n) \preceq n^b$.
\end{itemize}
Moreover, (i) holds if and only if $\Gamma$ is virtually nilpotent.
\end{cor}

Our second main result concerns the growth rate of $\Farb_{\Gamma,X}(n)$ and how it relates to the threshold between finite and infinite groups. It is straightforward to see that for an infinite group, $\word_{\Gamma,X}(n) \geq n$. However, the existence of infinite simple groups precludes such a growth threshold result for subgroup growth. As the function $\Farb_{\Gamma,X}$ relates these two functions, it is not clear if such a growth threshold result should hold for $\Farb_{\Gamma,X}$. That said, our final result exhibits that $\Farb_{\Gamma,X}$ does enjoy a growth threshold. Specifically,

\begin{thm}\label{FiniteTheorem}
Let $\Gamma$ be a finite generated group. If $\Farb_{\Gamma,X}(n) \preceq \log \log n$, then $\Gamma$ is finite.
\end{thm}

It was established in \cite[Lemma 1.1, Theorem 2.2]{B10} that if $\Gamma$ contains an element of infinite order, than $\log n \preceq \Farb_{\Gamma,X}(n)$. We give a slight improvement of Theorem \ref{FiniteTheorem} (see Scholium \ref{FiniteImprovement}) in Section \ref{FiniteSection}. The proof of Theorem \ref{FiniteTheorem} uses the above mentioned basic inequality relating $\Farb_{\Gamma,X}(n)$ with the word growth function $\word_{\Gamma,X}(n)$ and the normal subgroup growth function $\sub_\Gamma(n)$ established earlier in \cite[Equation 1]{BM11}.\smallskip\smallskip

\nid 
We conclude with some geometric motivation for the study of the functions $\D_\Gamma$, $\Farb_{\Gamma,X}$, and some related functions from \cite{BM11}.

%-----------------------------------------------------------
%-----------------------------------------------------------
\paragraph{\textbf{Acknowledgements.}} 
We are immensely grateful to the excellent referee for pointing out errors in an earlier draft of this paper. We thank Martin Kassabov for asking us a question that led us to find Theorem \ref{FiniteTheorem}. The first author was partially supported by NSF RTG grant DMS-0602191. The second author was partially supported by NSF DMS-1105710.

%-----------------------------------------------------------
%-----------------------------------------------------------
\paragraph{\textbf{Notation and Conventions.}} We write $f \preceq g$ to mean that there exists $C > 0$ such that $f(n) \leq C(g(Cn))$. If $f \preceq g$ and $g \preceq f$, then we write $f \approx g$. The growth of $\Farb_{\Gamma,X}(n)$ is, up to this equivalence, independent of $X$ (see \cite[Lemma 1.1]{B10}). Hence, we typically drop $X$ from the notation.

%-----------------------------------------------------------
%-----------------------------------------------------------
\section{A short algebraic excursion}

In the proof of Theorem \ref{PolyBound}, we require some results on divisibility functions for rings. This section contains a pair of lemmas for just this task. Throughout, $S$ will be either the ring $\Z[T]$ or $\mathbf{F}_p[T]$, where $T = \set{x_1,\dots,x_s}$ is a finite set of indeterminants. The divisibility function for $S$
\[ \D_S\colon S - \set{0} \lra \N \]
is given by
\[ \D_S(f) = \min\set{\abs{S/\Fr{p}}~:~ f \ne 0 \mod \Fr{p}, \:S/\Fr{p} \text{ is a field}}. \]
The next few results provide the needed control of this function in the characteristic zero and positive characteristic cases. 
We start with a lemma that allows us to reduce to the single variable case:

\begin{lemma} \label{LemmaZ}
Let $S = R[T]$ where $R = \mathbf{F}_p$ or $R = \Z$ and $T = \{ x_1, \ldots, x_s\}$.
Let $f \in S$ be a polynomial that is nonzero and of degree $d$.
Then there exists a sequence $\{n_i \}_{i=1}^s$ taking values in $\{ 0, 1, \ldots, d^{2s} \}$ such that
$$
f(x^{n_1}, \ldots, x^{n_s}) \neq 0.
$$
\end{lemma}

\begin{proof}
We prove this by complete 2-dimensional induction on $s$ and $d = \deg(f)$.
The base cases where $s = 1$ or $d = 0$ are trivial.
For the inductive step, let $f$ be a degree $d$ polynomial in $R[x_1, \ldots, x_s]$ and write
$$
f(x_1, \ldots, x_s) = (h_0 + x_1 h_1) x_1^k,
$$
where $h_0 \in R[x_2, \ldots, x_s]$ is nonzero, $h_1 \in R[x_1, \ldots, x_s]$, and $k$ a nonnegative integer.
If $k \neq 0$, we are done by the inductive hypothesis applied to $(h_0 + x_1 h_1)$, which has degree $< d$.
We assume, thusly, that $k = 0$.
Since $h_0$ is nonzero and in $R[x_2, \ldots, x_s]$ (note the variables start at $x_2$), there exists, by the inductive hypothesis, $n_2, \ldots, n_s \in \{ 0, 1, \ldots, d^{2s-2}\}$ such that
$$
h_0(x^{n_2}, \ldots, x^{n_s}) \neq 0.
$$
We are done if $h_1( x^{d^{2s}}, x^{n_2}, \ldots, x^{n_s}) = 0$ as then
$$
f(x^{d^{2s}}, x^{n_2}, \ldots, x^{n_s})  = h_0(x^{n_2}, \ldots, x^{n_s}) \neq 0.
$$
Otherwise $h_1( x^{d^{2s}}, x^{n_2}, \ldots, x^{n_s})$ is nonzero. In this case, we have
$$
\deg( h_0(x^{n_2}, \ldots, x^{n_s}) ) \leq d^{2s-1} < d^{2s} \leq \deg (x^{d^{2s}} h_1( x^{d^{2s}}, x^{n_2}, \ldots, x^{n_s})).
$$
Thus, 
$$
f(x^{d^{2s}}, x^{n_2}, \ldots, x^{n_s}) \neq 0
$$
as desired.
\end{proof}

Now we handle the characteristic zero case.

\begin{lemma} \label{LemmaA}
Let $S=\Z[x_1, \ldots, x_s]$, $f \in S$ with $\deg(f) \leq d$, and assume that $f$ is a nonzero function.
Set $g \in \Z[x]$ to be a single variable polynomial obtained by Lemma \ref{LemmaZ} and $\{ a_i \}$ be given by $g(x) = a_0 + a_1 x + \cdots + a_rx^r.$ Then for any $\epsilon > 0$,
\[ \D_S(f) \leq C \pr{\log\pr{\max\set{\abs{a_j}}} + d^{(2s+1)+\epsilon}}, \]
where $C$ depends only on $\Z$ and $s$.
\end{lemma}

\begin{proof}
Let $r = \deg (g)$ with $r \leq d^{2s+1}$. $g$ has at most $r$ roots and so there exists $\ell \in \N$ with $\ell \leq r+1$ such that $g(\ell) \neq 0$. Setting $g(\ell)=i$ and $A=\max\set{\abs{a_j}}$, note that $i \in \Z$ and $\abs{i} \leq (r+1)\ell^rA$. Since $\Farb_\Z(i) \approx \log(i)$ (see \cite[Theorem 2.2]{B10}), we see that
\[ \D_\Z(i) \leq C_0 \log((r+1)\ell^rA), \]
where $C_0$ is a constant that only depends on $\Z$. This inequality gives
\begin{align*}
\D_\Z(i) &\leq C_0 \pr{\log A + r\log \ell + \log(r+1)} \\
&\leq C_0\pr{\log A +d^{2s+1} \log(d^{2s+1}+1) + \log(d^{2s+1}+1)} \\
&\leq C_1 (\log A + d^{2s+1} \log d) < C ( \log A + d^{(2s+1)+\epsilon}),
\end{align*}
where $C_1$ and $C$ only depends on $\Z$ and $s$. In total, we have the sequence of ring homomorphisms
\[ S \lra \Z[x] \lra \Z \lra \B{F}_p \]
where
\[ f \longmapsto g \longmapsto i \longmapsto \overline{i} \ne 0. \]
In particular,
\[ \D_S(f) \leq \D_\Z(i) < C( \log A + d^{(2s+1)+\epsilon}). \]
\end{proof}

We next handle the positive characteristic case.

\begin{lemma} \label{LemmaB}
Let $S = \mathbf{F}_p[x_1, \ldots x_s]$, $f \in S$, $\deg(f) +1 \leq d$, and assume that $f$ is nonzero. Then
\[ \D_S (f) \leq d ^{C \log(p)}, \]
where $C$ depends only on $s$.
\end{lemma}

\begin{proof}
Set $g \in \mathbf{F}[x]$ to be a single variable polynomial obtained by Lemma \ref{LemmaZ}.
Then $g(x)$ is not the zero polynomial and $\deg(g) = r \leq d^{2s+1}$. Let $I_\ell(p)$ be the number of monic irreducible polynomials in $\B{F}_p[x]$ of degree equal to $\ell$. By a well-known result of Gauss (see for instance \cite[Corollary 9.2.3]{Rom}), we have
\[ I_\ell(p) = \frac{1}{\ell} \sum_{d \mid \ell} \mu(d)p^{\ell/d}, \]
where
\[ \mu(d) = \begin{cases} 1, & d = 1 \\ (-1)^k, & d = p_1\dots p_k,~p_j \text{ are distinct primes}, \\ 0, & \text{otherwise}. \end{cases} \]
In particular, for large values of $\ell$, we have that
\[ \frac{1}{2\ell} p^{\ell} \leq I_\ell(p) \leq 2 \frac{1}{\ell}p^\ell. \]
Hence $I_\ell(p) \geq p^{\ell/2}$ for sufficiently large $\ell$. This inequality in tandem with 
\[ \deg(g) \leq d^{2s+1} \] 
gives that there exists some polynomial $h \in I_{C' \log(d)}(p)$ where $h$ does not divide $g$ and $C'$ only depends on $s$. The quotient $\B{F}_p[x]/(h)$ has order less than or equal to $p^{C' \log(d)}$, and so we are done.
\end{proof}

%-----------------------------------------------------------
%-----------------------------------------------------------
\section{Proof of Theorem \ref{PolyBound}}

Before diving into the proof of Theorem \ref{PolyBound}, we give a brief sketch of the argument: for a finitely generated group, $\Gamma$, the field generated by the coefficients of the matrices over $\Q$ or $\mathbf{F}_p$ is finitely generated and so is a finite extension of a transcendental extension of some finite transcendence degree. Applying restriction of scalars (or corestriction), we can, at the cost of increasing the size of the matrices, assume the extension is purely transcendental. The coefficient ring generated over $\Z$ or $\mathbf{F}_p$ is then the ring $S$ from the previous section but with finitely many elements inverted. For a non-trivial element (which is represented by a matrix), we simply apply Lemma \ref{LemmaA} or \ref{LemmaB} to a non-trivial entry of that element after a scaling procedure. The map on matrices induced by the map of rings then provides us with a small finite quotient of $\Gamma$ that verifies the non-triviality of the given word. We then are able to write down bounds after relating the word length of the non-trivial word with the complexity of the non-trivial entry. With the sketch behind us, it is now time to dive.

\begin{proof}[Proof of Theorem \ref{PolyBound}]

Given a finitely generated group $\Gamma$ in $\GL(m,K)$ for an infinite field $K$, we select a finite generating set $X$ for $\Gamma$ and suppose further that the set $X$ generates $\Gamma$ as a monoid. 
For a given non-trivial word $\gamma \in \Gamma$, there is some coefficient $\gamma_{j,k} \in K$ that separates $\gamma$ from the identity matrix. The first step in our proof follows that of Mal'cev \cite{M40}.
Namely, we show that it suffices to restrict attention to  a subring of $K$ that is more tenable. 
To that end, let $L$ be the field generated by the (finitely many) entries that appear in the elements of $X$.
By construction, $L$ is finitely generated over $\Q$ or $\mathbf{F}_p$.
Moreover, the field $L$ is a finite extension of $\Q(T)$ or $\mathbf{F}_p(T)$, where $T = \set{x_1,\dots,x_s}$ is a transcendental basis (see, for instance, \cite[Corollary 3.3.3]{Rom}). By choosing a finite basis for $L$ as a vector space over $\Q(T)$ or $\mathbf{F}_p(T)$, we can embed $L$ into $\textrm{Mat}([L:\Q(T)],\Q(T))$ or $\textrm{Mat}([L:\mathbf{F}_p(T)],\mathbf{F}_p(T))$. Applying this embedding on the coefficients of the matrices in $\GL(m,L)$, we can view $\Gamma<\GL(m,L) < \GL(M,\Q(T))$ or $\GL(M,\mathbf{F}_p(T))$, where $M=m[L:\Q(T)]$ or $m[L:\mathbf{F}_p(T)]$. For each generator $\gamma_i \in X$ and each matrix coefficient $(\gamma_i)_{j,k}$, we have a finite number of elements in $\Z[T]$ or $\mathbf{F}_p[T]$ that are inverted; these are the elements in the denominators of the matrix coefficients of the generators. Ranging over all the generators and all of the matrix coefficients, we see that $\Gamma<\GL(M,S')$, where $S'$ is obtained from $S=\Z[T]$ or $\mathbf{F}_p[T]$, with a finite number of inverted elements. Note that in the case of $\Z[T]$, we, if necessary, invert some coefficients of $\Z$ along with some polynomials in these extended coefficients and so the ring is of the form $\Z[1/p_1,\dots,1/p_u][T]$ with a finite number of inverted primes in the coefficients and a finite number of inverted polynomials. In either the case of $\Z$ or $\mathbf{F}_p$, there exists $\Phi(T) \in S$ such that for each generator $\gamma_j$, $\Phi(T)I_M \gamma_j \in \GL(M,S)$. To obtain $\Phi(T)$, we can simply take the product of all of the denominators occurring in the coefficients $(\gamma_i)_{j,k}$. In the case of $\Z$, we, if necessary, multiply this product by some fixed integer to ensure the coefficients of the resulting polynomials are integer valued and also ensure that all of the primes in $\Z$ that are inverted are also in the product; this last demand will be useful later. We will continue throughout to denote by $S'$, the above ring with $\Gamma < \GL(M,S')$ and $\Phi(T) \in S$ such that $\Phi(T)I_M\gamma_j \in \GL(M,S)$. We further note that every unit in $S'$ can be generated multiplicatively by the various factors of $\Phi(T)$.

Let $w$ be the word that gives $\gamma$ in terms of the generators $X$. We will instead consider $A = \gamma - I_M$. Since $X$ generates $\Gamma$ as a monoid and $\Phi(T)I_M$ is central, we have that 
\[ w(\Phi(T) X) =  (\Phi(T))^{\norm{\gamma}_X} w(X). \] 
Hence, we can scale $A$ by $(\Phi(T))^{\norm{\gamma}_X}I_M$ so that the resulting element is in $\textrm{Mat}(M,S)$. An off-diagonal coefficient of $(\Phi(T))^{\norm{\gamma}_X}I_MA$ will be of the form
\[ ((\Phi(T))^{\norm{\gamma}_X}I_MA)_{i,j} = (\Phi(T))^{\norm{\gamma}_X}\gamma_{i,j}. \]
For any ring homomorphism $\vp\colon S' \to R$ where $R$ is a finite ring with identity and with
\[ \vp\pr{(\Phi(T))^{\norm{\gamma}_X}\gamma_{i,j}} \ne 0, \]
we must have $\vp(\gamma_{i,j}) \ne 0$ since $\Phi(T)$ is a unit in $S'$. The diagonal coefficients of $(\Phi(T))^{\norm{\gamma}_X}I_MA$ have the form
\[ ((\Phi(T))^{\norm{\gamma}_X}I_MA)_{i,i} = (\Phi(T))^{\norm{\gamma}_X}\gamma_{i,i} - (\Phi(T))^{\norm{\gamma}_X}. \]
For any ring homomorphism $\vp\colon S' \to R$ where $R$ is a finite ring with identity and with
\[ \vp\pr{(\Phi(T))^{\norm{\gamma}_X}\gamma_{i,i} - (\Phi(T))^{\norm{\gamma}_X}} \ne 0, \]
we must have
\[ \vp(\gamma_{i,i}-1) \ne 0 \]
since again $\Phi(T)$ is a unit in $S'$. In either case, the homomorphism
\[ \rho\colon \GL(M,S') \lra \GL(M,R) \]
induced by $\vp\colon S' \to R$ will have $\rho(\gamma) \ne 1$. Therefore, it suffices to find a ring homomorphism $\vp\colon S' \to R$ that does not kill all of the coefficients of $(\Phi(T))^{\norm{\gamma}_X}I_MA$. For this task, since these coefficients are in $S$, we can apply Lemma \ref{LemmaA} or \ref{LemmaB}. Note that since those lemmas have target rings $R$ that are finite fields and built into our assumptions, the image of $\Phi(T)$ must be non-zero (hence a unit), these homomorphisms for $S\to R$ extend to homomorphisms of $S' \to R$; this is why we insisted that $\Phi(T)$ involve enough units to generate the group of units of $S'$.

Let $A'$ be a non-zero coefficient of $(\Phi(T))^{\norm{\gamma}_X}I_MA$. In order to obtain quantified results, we must relate the word length of $\gamma$ to the degree of the $A'$. In the event $S= \Z[T]$, we must relate the word length of $\gamma$ to the maximum coefficients occurring in $A'$ as well. For the maximum coefficient control, it is straightforward to see that there exists a constant $\alpha$, depending on the generating set $X$ such that
\[ \alpha_{i,j} < \alpha^{\norm{\gamma}_X} \]
where $\alpha_{i,j}$ is maximum of the absolute values of the coefficients of $A'$; this fact was used previously in \cite{B10}. For the required degree control, there exists a constant $C_1$ that depends only on the generating set $X$ such that
\[ \deg(A') < C_1\norm{\gamma}_X. \]
The reason is identical to the coefficient control except now degree is additive under multiplication, thus yielding linear control opposed to exponential control. Note that this control on degree holds over both $\Z$ and $\mathbf{F}_p$. With these relationships established, we press forward, separating into two cases again. 

\textbf{Case 1.} $A' \in \Z[T]$.

By Lemma \ref{LemmaA}, we can find a map of $\Z[T]$ to a finite field $R$ with $\abs{R} \leq C(\log (\alpha^{\norm{\gamma}_X^2}) + \norm{\gamma}_X^{2s+2})$. Note that in using Lemma \ref{LemmaA}, we need control on the coefficients of 
\[ g(x)=A'(x^{n_1},x^{n_2},\dots,x^{n_s}), \]
where $n_i \leq \deg(A')^s$.
However, the maximum coefficient appearing in $g(x)$ is certainly no bigger than $\alpha^{\norm{\gamma}_X^2}$. So regardless of $A'$ being constant, we get an induced map of $\GL(M,S') \to \GL(M,R)$ has order at most $\abs{R}^{M^2}$. Since the coefficient $A'$ is not zero, the image of $\gamma$ is not trivial and so
\[ \D_{\Gamma}(\gamma) < C' \norm{\gamma}_X^{(2s+2)M^{2}}. \]

\textbf{Case 2.} $A' \in \B{F}_p[T]$.

The beginning of this case follows that of Case 1 with Lemma \ref{LemmaB} playing the role of Lemma \ref{LemmaA} (notice that the assumptions are slightly different). 
By increasing $C_1$ to a new constant $C_2$ (which depends only on $X$) we have
\[ \deg(A') + 1 < C_2 || \gamma ||_X \]
and so we are in a situation where Lemma \ref{LemmaB} applies. In either case, we obtain a field quotient of $S'$ to $R$ where $A'$ is not zero and $\abs{R} < C'\norm{\gamma}_X^{C'\log p}$ for a constant $C'$ depending on only on $X$ and $\abs{T}$. The induced map from $\GL(M,S') \to \GL(M,R)$ has order at most $\abs{R}^{M^2} <(C')^{M^2}\norm{\gamma}_X^{C'M^2\log p}$. As $\gamma$ is nontrivial under this homomorphism, we see that
\[ \D_{\Gamma}(\gamma) < C\norm{\gamma}_X^{CM^2} \]
for some constant $C$ independent of $\gamma$. In particular, in each case, we have
\[ \D_{\Gamma}(\gamma) < C\norm{\gamma}_X^{d} \]
for constants $d,C$ independent of $\gamma$. 

In both cases, we obtain the upper bound
\[ \Farb_{\Gamma,X}(n) \preceq n^d \]
for some constant $d$, as mandated by the theorem.
\end{proof}

%-----------------------------------------------------------
%-----------------------------------------------------------
\section{Proof of Theorem \ref{FiniteTheorem}}\label{FiniteSection}

We proceed via contradiction and assume that $\Gamma$ is infinite. Specifically, fixing a generating set $X$ for $\Gamma$, we assume both that $\Gamma$ is infinite and the inequality
\begin{equation}\label{6}
\Farb_{\Gamma,X}(n) \preceq \log\log(n)
\end{equation}
holds. With the aim of establishing a contradiction, we first note that
\begin{equation}\label{7}
n \preceq \word_{\Gamma,X}(n).
\end{equation}
Second, we have the basic inequality
\begin{equation}\label{8}
\log \word_{\Gamma,X}(n) \preceq \sub_\Gamma(\Farb_{\Gamma,X}(n)) \log \Farb_{\Gamma,X}(n)
\end{equation}
established in \cite[Equation 1]{BM10}. Note that this inequality holds for all generating sets $X$. Third, we have (see \cite[Proposition 2.8]{LS})
\begin{equation}\label{9}
\log \sub_\Gamma(n) \preceq (\log(n))^2.
\end{equation}
In total, these inequalities yield the following string
\begin{align*}
\log\log n &\preceq \log\log \word_{\Gamma,X}(n) \\
&\preceq \log(\sub_{\Gamma} (\Farb_{\Gamma,X}(n)) + \log\log \Farb_{\Gamma,X}(n)\\
&\preceq \left(\log(\Farb_{\Gamma,X}(n))\right)^2 + \log\log \Farb_{\Gamma,X}(n) \\
&\preceq (\log\log\log(n))^2,
\end{align*}
which is clearly impossible.\qed

As mentioned in the introduction, if $\Gamma$ contains an element of infinite order, according to \cite[Lemma 1.1, Theorem 2.2]{B10}, we have $\log(n) \preceq \Farb_{\Gamma,X}(n)$. Thus, the question of whether or not the above bound is optimal concerns only infinite, residually finite, torsion groups.

\begin{ques}\label{TorsionQues}
Does there exist an infinite, residually finite, torsion group $\Gamma$ with strict asymptotic inequalities
\[ \log \log (n) \prec \Farb_{\Gamma,X} \prec \log(n). \]
\end{ques}

One can certainly provide better lower bounds for $\Farb_{\Gamma,X}(n)$. If $x=x(n)=\log \Farb_{\Gamma,X}(n)$, we see from above that
\[ \log \log n \preceq x^2 + \log x. \]
In particular, so long as
\[ \limsup_{n \to \iny} \frac{x^2}{\log \log n} = 0, \]
we would derive a contradiction. Thus, we have:

\begin{sch}\label{FiniteImprovement}
If
\[ \limsup_{n \to \iny} \frac{(\log \Farb_{\Gamma,X}(n))^2}{\log \log n} = 0, \]
then $\Gamma$ is finite. In particular, $e^{\sqrt{\log \log n}} \preceq \Farb_{\Gamma,X}(n)$ if $\Gamma$ is infinite.
\end{sch}

An example of a faster growing function that satisfies the condition of Scholium \ref{FiniteImprovement} is
\[ \Farb(n) = (\log \log n)^{(\log \log \log n)^r}, \]
where $r>0$ is a fixed constant. However, we do not know of any examples of infinite, residually finite groups with strict asymptotic inequality $\Farb_{\Gamma,X}(n) \prec \log n$ and so feel Question \ref{TorsionQues} is interesting regardless of the lower bound on growth.

\section{Final remarks}

There is geometric motivation for our work here and in \cite{B10,B11,BM10,BM11}. For instance, let $\Gamma$ be the fundamental group of a closed $n$--manifold $M$ which admits a metric of negative curvature. We have a bijection between conjugacy classes in $\Gamma$ with closed geodesics on $M$. Moreover, by the \u{S}varc--Milnor Lemma, this bijection is bi-Lipschitz with respect to word and geodesic lengths. The function $\D_\Gamma(\gamma)$ provides the degree of the smallest regular cover where the geodesic corresponding to $\gamma$ fails to lift. By Theorem \ref{PolyBound}, the existence of a faithful linear representation affords one control over how big this degree can be as a function of the length of the geodesic. In addition, lower bounds on the function $\Farb_{\Gamma,X}$ give upper bounds on how quickly one can increase the systole of $M$ in finite regular covers. The growth threshold result, Theorem \ref{FiniteTheorem}, gives a uniform lower bound on the degree of the regular covers where a geodesic fails to lift. Moreover, results like Gromov's systolic inequality preclude one from growing the systole too quickly in finite covers, and the Girth inequality in \cite[Equation 2]{BM11} is analogous to a systolic inequality given the discussion here. It seems plausible that our work could be employed in systolic problems, though the fundamental group of the manifold would have more stringent restrictions than one might typically impose for these geometric problems. Consequently, the implementation of this ideology would likely only produce novel geometric results. We view this philosophical connection to be of greater interest.

%-----------------------------------------------------------
%-----------------------------------------------------------

%-----------------------------------------------------------
%-----------------------------------------------------------

\end{document}